\documentclass[preprint,12pt]{elsarticle}

\usepackage{amssymb}
\usepackage{amsmath,amssymb,amsthm}
\usepackage{lmodern}             
\usepackage{mathtools}
\usepackage{subcaption}
\newtheorem*{theorem*}{Theorem}

\newtheorem{lem}{Lemma}[subsection]% theorem counter resets every \subsection
\RequirePackage[colorlinks,citecolor=blue,urlcolor=blue]{hyperref}
\journal{}
\theoremstyle{remark}

\newcommand{\E}{\mathbb{E}}

\newcommand{\R}{\mathbb{R}}

\newcommand{\Z}{\mathbb{Z}}

\renewcommand{\AA}{\mathcal{A}}

\newcommand{\PP}{\mathcal{P}}

\newcommand{\eps}{\varepsilon}

\newcommand{\p}[1]{\left(#1 \right)}

\newcommand{\defeq}{\vcentcolon=}

\newcommand{\dd}{\mathrm{d}}

\newcommand{\fmin}{f_{\min}}
\newcommand{\fmax}{f_{\max}}

\begin{document}

\begin{frontmatter}



\title{A short proof on the rate of convergence of the empirical measure for the Wasserstein distance}


\author[1]{Vincent Divol}

\affiliation[1]{organization={Universit\'e Paris-Saclay and Inria Saclay},
country={France}}
\ead{vincent.divol@inria.fr}
%\ead[Homepage]{vincentdivol.github.io}

\begin{abstract}
\ %Assessing the rate of convergence of the empirical measure $\mu_n$ of a $n$-sample of law $\mu$ for the Wasserstein distance $W_p$ ($1\leq p \leq \infty$) is a fundamental question in statistical optimal transport.
 We provide a short proof that the Wasserstein distance between the empirical measure of a $n$-sample and the estimated measure is of order $n^{-1/d}$, if the measure has a lower and upper bounded density on the $d$-dimensional flat torus. 
 %when the measure i Using standard tools from Fourier analysis, we provide a short proof that the rate of convergence between is equal to $n^{-1/d}$ up to logarithmic factors, in the case where $\mu$ is assumed to have a lower and upper bounded density on the $d$-dimensional flat torus. 
\end{abstract}

% provide arXiv number if available:
%\arxiv{arXiv:0000.0000}

% put your definitions there:

\end{frontmatter}

 For $1\leq p < \infty$, let $W_p$ be the $p$-Wasserstein distance between measures, defined for two probability measures $\mu,\nu$ with finite $p$th moments supported on a metric space $(\Omega,\rho)$  by
\begin{equation}
W_p(\mu,\nu) \defeq \inf_{\pi\in \Pi(\mu,\nu)} C_p(\pi)^{1/p},
\end{equation}
where $\Pi(\mu,\nu)$ is the set of transport plans between $\mu$ and $\nu$, that is the set of probability measures on $\Omega\times \Omega$, with first marginal $\mu$ and second marginal $\nu$, and $C_p(\pi)=\iint \rho(x,y)^p \dd \pi(x,y)$ is the cost of the plan $\pi$. We define the distance $W_\infty$ by replacing the quantity $C_p(\pi)^{1/p}$ by the $\pi$-essential supremum of $\rho$. 
%The minimax rate of estimation for the estimation on $\mu$ with respect to $W_p$ on a certain class $\PP$ of probability distributions on $\Omega$ is by definition equal to
%\begin{equation}
%\RR_{n,p}(\PP) \defeq \inf_{\hat\mu_n} \sup_{\mu\in \PP} \E_{\mu^{\otimes n}} W_p(\hat\mu_n,\mu),
%\end{equation}
%where the infimum is taken over all estimators $\hat \mu_n$ of $\mu$. To put it another way, the minimax risk  is the best risk an estimator can attain uniformly on the class $\PP$ for the estimation of the measure $\mu$.

Let $\mu$ be a probability measure on some metric space $(\Omega,\rho)$, and let $\mu_n$ be the empirical measure associated to a $n$-sample $X_1,\dots,X_n$ of law $\mu$. The question of studying rates of convergence between $\mu$ and $\mu_n$ for Wasserstein distances $W_p$ has attracted a lot of attention over recent years (see e.g.~\cite{singh2018minimax,trillos2015rate}). % \cite{dereich2013constructive,fournier2015rate,kloeckner2018empirical,lei2020convergence, singh2018minimax,trillos2015rate,weed2019sharp}
 If no bounds on the density are assumed, then the quantity $\E W_p(\mu_n,\mu)$ is known to be bounded by a quantity of order $n^{-\frac{1}{2p}} + n^{-\frac{1}{d}}$ when $\Omega$ is a $d$-dimensional domain, and this bound is tight (see e.g~\cite{singh2018minimax}). For $p=\infty$, Nicol\'as Garc\' ia Trillos and Dejan Slep\v cev \cite{trillos2015rate} have shown that $\E W_\infty(\mu_n,\mu)$ is of order $(\log n/n)^{1/d}$ (for $d\geq 3)$ in the case where $\mu$ has a density $f$ which is lower bounded and upper bounded on some convex domain $\Omega$. As $W_p\leq W_\infty$, the same rate also holds for any $1\leq p \leq \infty$. This exhibits the following phenomenon: when $2p> d$, the problem of reconstructing $\mu$ for the Wasserstein distance is strictly harder if no bounds on the underlying density are assumed. 

In this note, we propose to give a short proof of the fact that $\E W_p(\mu_n,\mu)\lesssim n^{-1/d}$ (for $d\geq 3$) for bounded densities. We restrict to the case where $\Omega$ is the $d$-dimensional flat torus $\Omega$ in order to avoid complications due to boundary effects. Let $\PP_0$ be the set of probability distributions on $\Omega$, having a density $f$ satisfying $\fmin\leq f\leq \fmax$  for some $\fmax\geq \fmin>0$.
\begin{theorem*}\label{thm:main_thm}
Let $\mu\in \PP_0$ and $1\leq p <\infty$. Then, there exists a constant $C$ such that
\begin{equation}
\E W_p(\mu_n,\mu) \leq C\begin{cases} 
n^{-1/d} & \text{ if } d\geq 3,\\
(\log n)^{1/2}n^{-1/2}& \text{ if } d=2,\\
n^{-1/2} & \text{ if } d=1.
\end{cases}
\end{equation}
\end{theorem*}
The standard approach for bounding the distance $W_p(\mu_n,\mu)$ consists in precisely assessing the masses given by the measures $\mu_n$ and $\mu$ on dyadic partitions of the domain $\Omega$ (see e.g.~\cite{trillos2015rate}). We propose to take a different route by relying on a result from \cite{peyre2018comparison} which asserts that the Wasserstein distance is controlled by the pointed negative Sobolev distance when comparing measures having lower bounded densities. The proof is then completed by using tools from Fourier analysis. %Furthermore, the exponent in the logarithmic factor is slightly better for $d=2$, with the exponent $3/4$ being replaced by $1/2$.

We also note that minimax results from \cite{weed2019estimation} (proven for measures on the cube) can be straightforwardly adapted to the setting of the flat torus. In particular, those results imply that the rates exhibited in the theorem are optimal on the class $\PP_0$ (up to a logarithmic factor for $d= 2$).
%
%Such a remark was already made in  \cite{weed2019estimation}, where minimax rates of estimation are exhibited on class of densities of regularity $s$ on the cube, with different rates depending on whether the densities are assumed to have lower and upper bounds or not. 
%Let $\Omega_d= [0,1]^d/\Z^d$ be the flat $d$-dimensional torus endowed with its usual metric. Let $\PP_0$ be the set of probability distributions on $\Omega_d$, having a density $f$ satisfying $\fmin\leq f\leq \fmax$. Jonathan Niles-Weed and Quentin Berthet show in \cite{weed2019estimation} that $\RR_{n,p}(\PP_0)$ is of order $n^{-1/d}$ for $d\geq 2$ and $n^{-1/2}$ for $d=1$ (up to logarithmic factors). They also build a minimax estimator, which is built by projecting a wavelet density estimator on the set of probability distributions for the negative Besov norm $\|\cdot\|_{\BB^{-1}_{p,1}}$. In this short note, we show that the (much simpler) estimator $\mu_n$ is also a minimax estimator on $\PP_0$ (up to logarithmic factors).
%

%When $p$ is large, there is a great improvement between the previous known rate of $n^{-1/(2p)}$ and the optimal rate of $n^{-1/d}$. 

\section*{The proof} 
As $W_p\geq W_q$ if $p\geq q$, we may assume that $p\geq 2$. The proof of the theorem is heavily based on the following result of optimal transport theory, appearing in \cite{peyre2018comparison,nietert2021smooth}. Let $p^*$ be the conjugate exponent of $p$. For $\phi\in L_p$ with $\int \phi=0$, introduce the pointed negative Sobolev norm
\begin{equation}
\|\phi\|_{\dot H_p^{-1}} \defeq \sup\left\{\int \phi\psi,\ \|\nabla \psi\|_{L_{p^*}} \leq 1\right\},
\end{equation}
where the supremum is taken over all smooth functions $\psi$ defined on $\Omega$.% and $\|u\|_{L_{p^*}} \defeq \sum_{i=1}^d \|u_i\|_{L_{p^*}}$ for $u:\Omega\to \R^d$.
\begin{lem}\label{lem:santam}
Let $\mu,\nu$ be two measures on $\Omega$ having densities $f,g$. Assume that $f\geq \fmin$. Then,
\begin{equation}
W_p(\mu,\nu) \leq p\fmin^{1/p-1}\|f-g\|_{\dot H_p^{-1}}.
\end{equation}
\end{lem}

 Let $K$ be a smooth radial nonnegative function with $\int K=1$, supported on the unit ball and, for $h>0$, let $K_h=h^{-d}K(\cdot/h)$. Let $\mu_{n,h}$ be the measure having density $K_h*\mu_n$ on $\Omega$, i.e.~the density at a point $x\in \Omega$ is given by $f_{n,h}(x)\defeq \sum_{j=1}^n K_h(x-X_j)/n$.
% The term $\rho_h(X_i)$ is a boundary correction term, which ensures that $\mu_{n,h}$ is indeed a probability measure. Note in particular that $\rho_h(X_i)=1$ if $X_i$ is at distance more than $h$ from the boundary of $\Omega$, and that $\rho_h(X_i)\geq \int_{\Omega^+} K(x)\dd x \eqdef c_0 >0$, where $\Omega^+$ is the upper quadrant in $\R^d$.

\begin{lem}\label{lem:approximation}
We have $W_p(\mu_n,\mu_{n,h}) \leq C_0 h$, where $C_0= \p{\int |x|^p K(x)\dd x}^{1/p}$.
\end{lem}
\begin{proof}
Consider the unique transport plan $\pi_j$ between $K_h*\delta_{X_j}$ and $\delta_{X_j}$. The cost of $\pi_j$ is equal to $ \int |x-X_j|^p K_h(x-X_j)\dd x =h^p \int |x|^p K(x)\dd x$. The measure $\frac{1}{n}\sum_{j=1}^n\pi_j$ is a transport plan between $\mu_{n,h}$ and $\mu_n$, with associated cost equal to $h^p\int |x|^p K(x)\dd x$.
\end{proof}
By Lemmas \ref{lem:santam} and \ref{lem:approximation}, 
\begin{equation}\label{eq:first}
\begin{split}
\E W_p(\mu_n,\mu)&\leq \E W_p(\mu_n,\mu_{n,h}) +\E W_p(\mu_{n,h},\mu) \\
&\leq C_0h+p\fmin^{1/p-1}\E  \|f_{n,h}-f\|_{\dot H_p^{-1}}.
\end{split}
\end{equation}
To further bound this quantity, we use the following relation between the negative Sobolev norm and the Fourier decomposition of a signal.
 Given $\phi\in L_p$, we let $\hat\phi$ be the sequence of Fourier coefficients of $\phi$ (indexed by $\Z^d$) and denote by $ ^\vee$ the inverse Fourier transform. Let $|x|\defeq \sum_{i=1}^d|x_i|$ for $x\in \R^d$. A multiplier $s$ is a bounded sequence indexed by $\Z^d$ such that the operator $\phi\in L_p\mapsto (s\hat\phi)^{\vee}\in L_p$ is bounded. A sufficient condition for a sequence to be a multiplier is given by  Mikhlin multiplier theorem \cite[Theorem 3.6.7, Theorem 5.2.7]{Grafakos2008}.

\begin{lem}\label{lem:mikhlin}
Let $s:\R^d\to \R$ be a smooth function such that $|\partial^\alpha s(\xi)|\leq B|\xi|^{-|\alpha|}$ for every multiindex $\alpha$ with $|\alpha|\leq d/2+1$. Then, the sequence $(s(m))_{m\in\Z^d}$ is a multiplier with corresponding operator of norm smaller than $C_{p,d}B$.
\end{lem}
%Let $\Z^d_0 = \Z^d\backslash\{0\}$ and let $\phi\in L_p(\Omega)$ with $\int\phi=0$. For $m\in \Z^d$, define the function $e_m:x\in \Omega\mapsto e^{2i\pi  m\cdot x}$ and let $\hat \phi(m) = \int \phi\overline e_m$. In particular $\hat\phi(0)=0$, and $\phi = \sum_{m\in\Z^d_0}\hat\phi(m) e_m$ with convergence holding in $L_p(\Omega)$. 
%For $\xi\in\R^d$, let $|\xi|=\sum_{i=1}^d |\xi_i|$. 
Let $a:\R^d\to \R$ be a smooth function with $a(\xi) = 1/|\xi|$ for $|\xi|\geq 1$ and $a(0)=0$. Let $\AA$ be the associated multiplier operator (by Lemma \ref{lem:mikhlin}) defined by $\AA(\phi) = (a \hat\phi)^\vee$.
\begin{lem}\label{lem:riesz}
Let $\phi\in L_p$ with $\int \phi=0$. Then, $\|\phi\|_{\dot H_p^{-1}}\leq C_5\|\AA(\phi)\|_{L_p}$.
\end{lem}
\begin{proof}
Let $\psi:\Omega\to \R$ be a smooth function with $\|\nabla \psi\|_{L_{p^*}}\leq 1$. As $\hat\phi(0)=0$, we have \[\int \phi \psi = \sum_{m\in \Z^d} \hat\phi(m)\hat \psi(m) = \sum_{m\in \Z^d} a(m)\hat \phi(m)|m| \hat \psi(m) \leq \|\AA(\phi)\|_{L_p} \|(|\cdot|\hat \psi)^\vee\|_{L_{p^*}}.\] 
Note that $|\cdot|= \sum_{i=1}^d \eps_i e_i$, where $e_i(m)=m_i$ and $\eps_i(m)$ is the sign of $m_i$. As $\eps_i$ is a multiplier (by Lemma \ref{lem:mikhlin}), we have $\|(|\cdot|\hat \psi)^\vee\|_{L_{p^*}}\leq c\sum_{i=1}^d \|(e_i\hat \psi)^\vee\|_{L_{p^*}} = c\sum_{i=1}^d \|\partial_i \psi\|_{L_{p^*}} \leq  C_5$. 
%Let $g:\Omega\to \R$ be a differentiable function with $\|\nabla g\|_{L_{p^*}}\leq 1$. The function $u=(|\cdot|^{-2}\hat\phi)^\vee$ satisfies $\Delta u=\phi$ (weakly). Therefore, we have $\int \phi g=-\int \nabla u\cdot \nabla g\leq \|\nabla u\|_{L_p}$. As $\AA(\phi)=\nabla u$, we have $\|\phi\|_{\dot H_p^{-1}}\leq \|\AA(\phi)\|_{L_p}$. %To  prove the other inequality, choose $g=-u/\|\nabla u\|_{L_{p^*}}$.
\end{proof}
Hence, to conclude, it suffices to bound $\E\|\AA(f_{n,h}-f)\|_{L_p}\leq  \|\AA(f_h-f)\|_{L_p} + \E\|\AA(f_{n,h}-f_h)\|_{L_p}.$

\paragraph{Bound of the bias} Let $\kappa$ be the Fourier transform of $K$. As $K$ is smooth and compactly supported, $\kappa$ is a multiplier by Lemma \ref{lem:mikhlin}. Also, the function $M=a\cdot(\kappa-1)$ is a multiplier as a product of multiplier. Remark that $\hat f_h-\hat f= (\kappa(h\hspace{.05cm}\cdot\hspace{.05cm})-1)\hat f$, so that $\AA(f_h-f) = h(M(h \hspace{.05cm}\cdot\hspace{.05cm})\hat f)^{\vee}$. As the multiplier norms of $M$ and $M(h\hspace{.05cm}\cdot\hspace{.05cm})$ are equal \cite[Theorem 3.6.7]{Grafakos2008}, we have 
\begin{equation}\label{eq:bias}
 \|\AA(f_h-f)\|_{L_p} \leq hC_6\|f\|_{L_p}  \leq hC_6\fmax.
\end{equation}

\paragraph{Bound of the fluctuations} Eventually, we bound 
\begin{equation}
\E\|\AA(f_{n,h}-f_h)\|_{L_p} \leq \E\left[\|\AA(f_{n,h}-f_h)\|_{L_p}^p\right]^{1/p}.
\end{equation} 
 The random variable $\AA(f_{n,h})$ is equal to $n^{-1}\sum_{j=1}^n U_j$, where $U_j  \defeq \AA(K_h*\delta_{X_j}) = \AA(K_h)(\hspace{.05cm}\cdot-X_j)$ and $\E U_j = \AA(f_h)$.
% = \sum_{m\in \Z^d_0} a(hm)\kappa(hm) e^{-2i\pi X_j\cdot m}e_m(x)
We control the expectation of the $L_p$-norm of the sum of i.i.d.~centered functions thanks to the next lemma, which is a direct consequence of Rosenthal inequality \cite{rosenthal1970subspaces}.
\begin{lem}\label{lem:rosenthal}
Let $U_1,\dots,U_n$ be i.i.d.~functions on $L_p$. Then, the expectation $\E \left\|\frac{1}{n}\sum_{i=1}^n (U_i-\E U_i)\right\|_{L_p}^p$ is smaller than
\begin{equation}\label{eq:rosenthal}
C_pn^{-p/2} \int \p{\E |U_1(x)|^2}^{p/2}\dd x + C_p n^{1-p}\int \E \left[|U_1(x)|^p\right]\dd x.
\end{equation}
\end{lem}
Let $v_h$ be the sequence in $\ell_{p^*}(\Z^d)$ defined by $v_h(m) = a(m)\kappa(hm)$ for $m\in \Z^d$. By a change of variable, we obtain
\begin{equation}
\E \left[|U_1(x)|^p\right] = \int f(y)|\AA(K_h)(x-y)|^p\dd y \leq \fmax \|\AA(K_h)\|_{L_p}^p\leq \fmax \|v_h\|_{\ell_{p^*}}^p,
\end{equation}
where, at the last line, we applied Hausdorff-Young inequality \cite[Section XII.2]{zygmund2003trigonometric}. The last step consists in bounding $\|v_h\|_{\ell_{p^*}}^{p^*}$. We separate this quantity into two parts: $S_0 = \sum_{|hm|\leq 1} |v_h(m)|^{p^*}$ and $S_1 = \sum_{|hm|> 1} |v_h(m)|^{p^*}$. To bound $S_0$, we use that $\kappa$ is bounded on the unit ball, so that $S_0$ is of the order 
\begin{equation}\label{eq:controlS_0}
\sum_{|hm|\leq 1}|m|^{-p^*}\lesssim \begin{cases}
h^{p^*-d} &\text{ if } d\geq 3 \text{ or } (d=2 \text{ and } p>2)\\
-\log h &\text{ if } p=d=2\\
1 & \text{ if } d=1.
\end{cases}
\end{equation}
To bound $S_1$, we use that $|\kappa(hm)|\leq C_\gamma|hm|^{-\gamma}$ for any $\gamma>0$. Choosing $\gamma$ such that $\gamma p^*+p^*>d$, we obtain that $S_1$ is of the order 
\begin{equation}\label{eq:controlS_1}
h^{-\gamma p^*}\sum_{|hm|> 1}|m|^{-\gamma p^*-p^*} \lesssim h^{p^*-d}.
\end{equation}
Putting together inequalities \eqref{eq:rosenthal}, \eqref{eq:controlS_0} and \eqref{eq:controlS_1} yields that, for $h$ of the order $n^{-1/d}$, the expectation $\E\|\AA(f_{n,h}-f_h)\|_{L_p}$ is of the order
\begin{equation}\label{eq:var}
\begin{cases}
h/\sqrt{nh^d} \lesssim n^{-1/d} & \text{ if } d\geq 3,\\
\sqrt{(-\log h)/n}\lesssim (\log n)^{1/2}n^{-1/2} &\text{ if } d=2,\\
n^{-1/2} &\text{ if } d=1.
\end{cases}
\end{equation}
We conclude the proof by putting together the estimates \eqref{eq:first}, \eqref{eq:bias} and \eqref{eq:var}.

\paragraph{Remark 1}
For $p=2$, Mikhlin multiplier theorem can be replaced by Parseval's theorem, further simplifying the proof.

\paragraph{Remark 2}
A similar proof shows that the risk of the measure $\mu_{n,h}$ satisfies $\E W_p(\mu_{n,h},\mu)\lesssim n^{-(s+1)/(2s+d)}$ if $f$ is assumed to be of regularity $s$. Indeed, we can exploit the regularity of $s$ to show that, if $\kappa$ has sufficiently many zero derivatives at $0$, then the bias term is of order $h^{s+1}$, while the fluctuation terms is bounded in the same way. We then obtain the desired rate by choosing $h$ of the order $n^{-1/(2s+d)}$. This rate is in accordance with the minimax result of \cite{weed2019estimation}, where a modified wavelet density estimator is shown to attain the same rate of convergence.

\bibliographystyle{plain}
\bibliography{biblio.bib}

\begin{thebibliography}{1}

\bibitem{Grafakos2008}
Loukas Grafakos.
\newblock {\em Classical {F}ourier analysis}, volume~2.
\newblock Springer.

\bibitem{nietert2021smooth}
Sloan Nietert, Ziv Goldfeld, and Kengo Kato.
\newblock From smooth {W}asserstein distance to dual {S}obolev norm: Empirical
  approximation and statistical applications, 2021.

\bibitem{peyre2018comparison}
R{\'e}mi Peyre.
\newblock Comparison between {$W_2$} distance and {$\dot H_{-1}$} norm, and
  localization of {W}asserstein distance.
\newblock {\em ESAIM. Control, Optimisation and Calculus of Variations}, 24(4),
  2018.

\bibitem{rosenthal1970subspaces}
Haskell~P Rosenthal.
\newblock On the subspaces of {$L_p$} ({$p> 2$}) spanned by sequences of
  independent random variables.
\newblock {\em Israel Journal of Mathematics}, 8(3):273--303, 1970.

\bibitem{singh2018minimax}
Shashank Singh and Barnab{\'a}s P{\'o}czos.
\newblock Minimax distribution estimation in {W}asserstein distance.
\newblock {\em arXiv preprint arXiv:1802.08855}, 2018.

\bibitem{trillos2015rate}
Nicol{\'a}s~Garcia Trillos and Dejan Slep{\v{c}}ev.
\newblock On the rate of convergence of empirical measures in
  {$\infty$}-transportation distance.
\newblock {\em Canadian Journal of Mathematics}, 67(6):1358--1383, 2015.

\bibitem{weed2019estimation}
Jonathan Weed and Quentin Berthet.
\newblock Estimation of smooth densities in {W}asserstein distance.
\newblock In {\em Conference on Learning Theory}, pages 3118--3119, 2019.

\bibitem{zygmund2003trigonometric}
A.~Zygmund and R.~Fefferman.
\newblock {\em Trigonometric Series}.
\newblock Cambridge Mathematical Library. Cambridge University Press, 2003.

\end{thebibliography}

\end{document}